\documentclass[12 pt,english,oneside]{amsart}
\usepackage[top=1.16in, bottom=1.16in, left=1.23in, right=1.23in]{geometry}
\usepackage[latin1]{inputenc}
\usepackage[T1]{fontenc}
\usepackage[normalem]{ulem}
\usepackage{verbatim} 
\usepackage{graphicx}
\usepackage{xypic} 
\usepackage{amsfonts}
\usepackage{amsmath}
\usepackage{amssymb}
\usepackage{graphicx,epsfig}
\usepackage{amsthm}
\usepackage{amscd}
\usepackage{dsfont}
\usepackage{fancyhdr}

\newcommand{\g}{\mathfrak{g}}

\newcommand{\ra}{\rightarrow}

\newcommand{\ve}{\varepsilon}
\newcommand{\vp}{\varphi}

\newcommand{\ts}{\otimes}
\newcommand{\s}{\sigma}

\newcommand{\mc}{\mathbb{K}}
\newcommand{\fb}{\mathfrak{B}}

\newcommand{\p}{\partial}
\newcommand{\wt}{\widetilde}

\newcommand{\id}{\operatorname*{id}}
\newcommand{\dd}{\operatorname*{d}}

\newcommand{\Rel}{\operatorname*{Rel}}
\newcommand{\Con}{\operatorname*{Con}}
\newcommand{\Span}{\operatorname*{span}}
\newcommand{\CL}{\operatorname*{(}}
\newcommand{\im}{\operatorname*{im}}
\newcommand{\CR}{\operatorname*{)}}
\newcommand{\ideal}{\operatorname*{ideal}}
\newcommand{\End}{\operatorname*{End}}

\newtheorem{definition}{Definition}
\newtheorem{proposition}{Proposition}
\newtheorem{theorem}{Theorem}
\newtheorem{lemma}{Lemma}
\newtheorem{corollary}{Corollary}
\newtheorem{remark}{Remark}
\newtheorem{example}{Example}

\title{Non-symmetrizable quantum groups: defining ideals and specialization}
\author{Xin FANG}
\address{D\'epartement de Math\'ematiques, B\^atiment 425, Facult\'e des Sciences d'Orsay, Universit\'e Paris-Sud, F-91405 Orsay Cedex, France.}
\email{xinfang.math@gmail.com}

\begin{document}
\maketitle
\begin{abstract}
Two generating sets of the defining ideal of a Nichols algebra of diagonal type are proposed, which are then applied to study the bar involution and the specialization problem of quantum groups associated to non-symmetrizable generalized Cartan matrices.
\end{abstract}

\section{Introduction}

\subsection{Motivations}
Quantized enveloping algebras (quantum groups) $U_q(\g)$ are constructed by V. Drinfeld and M. Jimbo in the eighties of the last century by deforming the usual enveloping algebras associated to symmetrizable Kac-Moody Lie algebras $\g$ in aim of finding solutions of the Yang-Baxter equation. They motivate numerous work in the last three decades such as pointed Hopf algebras, canonical (crystal) bases, quantum knot invariants, quiver representations and Hall algebras, (quantum) cluster algebras, Hecke algebras, quantum affine and toroidal algebras, and so on.
\par
In the original definition of a quantum group in generators and relations, the symmetrizable condition on the Cartan matrix is essential in writing down explicitly the quantized Serre relations. With this explicit expression, it is not difficult to construct a specialization map \cite{Lus88} sending the quantum parameter $q$ to $1$ to recover the enveloping algebra $U(\g)$, which is shown to be an isomorphism of Hopf algebra. It should be remarked that the well-definedness of the specialization map depends on the knowledge of the quantized Serre relations and the Gabber-Kac theorem \cite{GK81} in Kac-Moody Lie algebras.
\par
In a survey article \cite{Kas95}, M. Kashiwara asked the following question: has a crystal graph for non-symmetrizable $\g$ a meaning? He also remarked that the definition of the quantum group $U_q(\g)$ associated to an arbitrary Kac-Moody Lie algebra $\g$ is not known at that time.
\par
This problem was recently solved in the combinatorial level by Joseph and Lamprou \cite{Joseph}: they constructed the abstract crystal $\mathcal{B}(\infty)$ associated to a generalized Cartan-Borcherds matrix (not necessary symmetrizable) without passing to the quantized enveloping algebra but adopting the path model construction after Littelmann \cite{Lit95} by using the action of root operators on a sort of good paths. This construction is combinatorial and it is natural to ask for a true algebra bearing it and the globalization of these local crystals. This is the main motivation of our study on the non-symmetrizable quantum groups. This project is divided into three steps:
\begin{enumerate}
\item define the quantum group associated to a non-symmetrizable generalized Cartan matrix, study their structures, specializations and the existence of the bar involutions;
\item define the $q$-Boson algebra, its action on the negative part of this quantum group and its semi-simplicity; then define the Kashiwara operators associated to simple roots;
\item establish the local crystal structure and its globalization, compare the former with the construction of Joseph-Lamprou.
\end{enumerate}
In this paper we will tackle the first step. The second step is almost achieved, where the main tool is the construction in \cite{Fang10}, details will be given in a consecutive paper.
\par
The first functorial (coordinate-free) construction of (the positive or negative part of) the quantum group appears in the work of M. Rosso \cite{Ros95}, \cite{Ros98} with the name "quantum shuffle algebras" and then interpreted in a dual language by Andruskiewitsch and Schneider \cite{AS02} named "Nichols algebras". These constructions largely generalize the definition of the usual quantum group and can be applied in particular to the non-symmetrizable case to obtain a half of the quantum group. The quantum double construction can be then applied to combine the positive and the negative parts to yield the whole quantum group.
\par
As a summary, we can associate a Hopf algebra (the quantum double of the bosonized Nichols algebra) to a generalized Cartan matrix $C$ which is not necessary symmetrizable. It is natural to ask whether there exists a specialization map from this Hopf algebra to the enveloping algebra of the Kac-Moody Lie algebra $\g(C)$ associated to $C$: this is not easy since in general, both the Nichols algebra and the Kac-Moody Lie algebra do not admit explicit presentations by generators and relations.
\par
The goal of this paper is twofold: on one hand, tackling the specialization problem in the non-symmetrizable case by the study of the defining ideal of the corresponding Nichols algebra; on the other side, defining the bar involution in the non-symmetrizable case. As a byproduct, we get an estimation on the size of the defining ideal.

\subsection{Defining ideals in Nichols algebras}
Let $(V,\s)$ be a braided vector space. The tensor algebra $T(V)$ admits a braided Hopf algebra structure by imposing a coproduct making elements in $V$ primitive; it can be then extended to the entire $T(V)$ in replacing the usual flip by the braiding.
\par
If the braiding $\s$ arises from an $H$-Yetter-Drinfel'd module structure on $V$ over a Hopf algebra $H$, the Nichols algebra can be defined as the quotient of $T(V)$ by some maximal ideal and coideal $\mathfrak{I}(V)$ contained in the set of elements of degree no less than $2$. We call $\mathfrak{I}(V)$ the defining ideal of the Nichols algebra $\mathfrak{B}(V)$.
\par
As an example, for a symmetrizable Kac-Moody Lie algebra $\g$, the negative part $U_q^-(\g)$ of the corresponding quantum group is a Nichols algebra, in which case the defining ideal $\mathfrak{I}(V)$ is generated as a Hopf ideal by quantized Serre relations. In general, it is very difficult to find out a minimal generating set of $\mathfrak{I}(V)$ as a Hopf ideal in $T(V)$. 
\par
In \cite{A03}, Andruskiewitsch asked some questions which guide the researches of this domain and the following ones concerning defining ideals appear therein:
\begin{enumerate}
\item For those $\mathfrak{B}(V)$ having finite Gelfan'd-Kirillov dimension, decide a minimal generating set of $\mathfrak{I}(V)$.
\item When is the ideal $\mathfrak{I}(V)$ finitely generated?
\end{enumerate}

The first general result on the study of the defining ideal is due to M. Rosso \cite{Ros98} and P. Schauenburg \cite{Sbg96}: they characterize it as the kernel of the total symmetrization operator. Recently, for Nichols algebras of diagonal type with finite root system, a minimal generating set of the defining ideal is found by I. Angiono. In this case, the corresponding Lyndon words and their symmetries (Lusztig's isomorphisms) \cite{Hec06} play an essential role.
\par
In \cite{Fang11}, we proposed the notion of "level $n$" elements with the help of a decomposition of the total symmetrization operators in the braid groups and proved their primitivity. These elements could be easily computed and the degrees where they appear are strongly restricted. This construction demands no concrete restriction on the braiding hence quite general, but we must pay the price that they may not generate the defining ideal.
\par
Once restricted to the diagonal case where the braiding is a twist by scalars of the usual flip, with some modifications on the conditions posed on "level $n$" elements, we obtain a generating set formed by some "pre-relations".

\subsection{Main ideas and results}
The main part of this paper is devoted to propose some methods to study a slightly modified version of the above problems. First, we will restrict ourselves to the infinite dimensional Nichols algebras of diagonal type having not necessarily finite Gelfan'd-Kirillov dimensions. Second, our principle has a pragmatic feature: we do not always desire a minimal generating set of the defining ideal but are satisfied with finding generating subsets fitting for solving concrete problems.
\par
We propose four subsets of the defining ideal $\mathfrak{I}(V)$: left and right constants, left and right pre-relations. The first two sets are defined as the intersection of kernels of left and right differential operators and the last two are their subsets obtained by selecting elements which are contained into the images of the Dynkin operators. Two main results (Theorem \ref{Thm:1} and Theorem \ref{Relr}) of this paper state that both are generators of the defining ideal.
\par
These results are then applied to the study of the specialization problem. In general, if the generalized Cartan matrix $C$ is not symmetrizable, we show in a counterexample that the natural specialization map may not be well-defined. Therefore in our approach, the first step is to pass to a symmetric matrix $\overline{C}$ by taking the average of the Cartan matrix. A result due to Andruskiewitsch and Schneider ensures that this procedure does not lose too much information.
\par
Once passed to the averaged matrix, we prove in Theorem \ref{Thm:3} that the specialization map $U_q(\overline{C})\ra U(\g(\overline{C}))$ is well-defined and is surjective. 
\par
As another application, we relate the degrees where pre-relations may appear with integral points of some quadratic forms arising from the action of the centre of the braid group. This allows us 
\begin{enumerate}
\item to reprove some well-known results in a completely different way which we hope could shed light on the finite generation problem of $\mathfrak{I}(V)$;
\item to explain that the set of left and right pre-relations are not too large.
\end{enumerate}

\subsection{Constitution of this paper}

After giving some recollections on Nichols algebras and braid groups in Section \ref{Sec5.2} and \ref{Sec5.3}, we define the constants and pre-relations in Section \ref{Sec5.4} and \ref{Sec5.5} and show that they are indeed generating sets. These results are then applied to study the specialization problem in Section \ref{Sec5.6} and \ref{Sec5.7}. Another application to the finitely generating property is given in Section \ref{Sec5.8}.

\subsection{Acknowledgements}
This paper is extracted (with revision) from the author's Ph.D thesis, supervised by Professor Marc Rosso, to whom I would like to express my sincere gratitudes. I would like to thank the referee for the valuable comments and suggestions to improve this paper.

\section{Recollections on Nichols algebras}\label{Sec5.2}

Let $\mc$ be a field of characteristic $0$ and $\mc^\times=\mc\backslash\{0\}$. All algebras and vector spaces, if not specified otherwise, are over the field $\mc$.

\subsection{Nichols algebras}

Let $H$ be a Hopf algebra and ${}^H_H\mathcal{YD}$ be the category of $H$-Yetter-Drinfel'd modules. The category ${}^H_H\mathcal{YD}$ is a braided category: for any $V,W\in{}^H_H\mathcal{YD}$, we let $\s_{V,W}:V\ts W\ra W\ts V$ denote the braiding. With this notation, $(V,\s_{V,V})$ is a braided vector space. Readers unfamiliar with these constructions are sent to \cite{AS02} for a survey.

\begin{definition}[\cite{AS02}]
A graded braided Hopf algebra $R=\bigoplus_{n=0}^\infty R(n)$ is called a Nichols algebra of $V\in{}_H^H\mathcal{YD}$ if 
\begin{enumerate}
\item $R(0)\cong \mc$, $R(1)\cong V$;
\item $R$ is generated as an algebra by $R(1)$;
\item $R(1)$ is the set of all primitive elements of $R$.
\end{enumerate}
We let $\mathfrak{B}(V)$ denote this braided Hopf algebra.
\end{definition}

The Nichols algebra $\mathfrak{B}(V)$ can be realized concretely as a quotient of the braided tensor Hopf algebra $T(V)$.

\begin{remark}[\cite{AS02}]\label{example}
Let $V\in{}_H^H\mathcal{YD}$ be an $H$-Yetter-Drinfel'd module. There exists a braided tensor Hopf algebra structure on the tensor space 
$$T(V)=\bigoplus_{n=0}^\infty V^{\ts n}.$$
\begin{enumerate}
\item The multiplication on $T(V)$ is given by the concatenation.
\item The coalgebra structure is defined on $V$ by: for any $v\in V$, $\Delta(v)=v\ts 1+1\ts v$, $\ve(v)=0$. Then it can be extended to the whole $T(V)$ by the universal property of $T(V)$ as an algebra.
\end{enumerate}
For $k\geq 2$, let $T^{\geq k}(V)=\bigoplus_{n\geq k}V^{\ts n}$
and $\mathfrak{I}(V)$ be the maximal coideal of $T(V)$ contained in $T^{\geq 2}(V)$: it is also a two-sided ideal (\cite{AS02}); the Nichols algebra $\mathfrak{B}(V)$ associated to $V$ is isomorphic to $T(V)/\mathfrak{I}(V)$ as a braided Hopf algebra. We let $S$ denote the convolution inverse of the identity map on $\mathfrak{B}(V)$.
\end{remark}

\begin{remark}
The construction of a Nichols algebra $\mathfrak{B}(V)$ is still valid when $(V,\s)$ is a braided vector space.
\end{remark}

\subsection{Nichols algebras of diagonal type}

\begin{definition}[\cite{AS02}]
The Nichols algebra $\mathfrak{B}(V)$ associated to a braided vector space $(V,\s)$ is called of diagonal type if there exists a basis $\{v_1,\cdots,v_N\}$ of $V$ and a matrix $(q_{ij})_{1\leq i,j\leq N}\in M_N(\mc^\times)$ of non-zero scalars such that for any $1\leq i,j\leq N$, $\s(v_i\otimes v_j)=q_{ij} v_j\otimes v_i$. The scalar matrix is called the braiding matrix of $(V,\s)$.
\end{definition}

In the situation of Remark \ref{example}, we will abuse to say that $T(V)$ is of diagonal type if $\mathfrak{B}(V)$ is so.
\par
The following example of the Nichols algebra of diagonal type is the main object we will study in this paper. Let $G=\mathbb{Z}^N$ be the additive group, $H=\mc[G]$ be its group algebra and $\widehat{G}$ be the character group of $G$. Let $V\in{}_H^H\mathcal{YD}$ be an $H$-Yetter-Drinfel'd module of dimension $N$; it admits a decomposition into linear subspaces $V=\bigoplus_{g\in G}V_g$ where $V_g=\{v\in V|\ \delta(v)=g\ts v\}$, here $\delta:V\ra H\ts V$ is the comodule structure map; moreover, there exist a basis $\{v_1,\cdots,v_N\}$ of $V$,  elements $g_1,\cdots,g_N\in G$ and characters $\chi_1,\cdots,\chi_N\in\widehat{G}$ such that $v_i\in V_{g_i}$ and for any $g\in G$,
$$g.v_i=\chi_i(g)v_i.$$
\par
In this case the braiding $\s_{V,V}$ has the following explicit form: for $1\leq i,j\leq N$,
$$\sigma_{V,V}(v_i\ts v_j)=\chi_j(g_i)v_j\ts v_i.$$
Therefore the Nichols algebra associated to $(V,\s_{V,V})$ is of diagonal type with braiding matrix $(q_{ij})_{1\leq i,j\leq N}=(\chi_j(g_i))_{1\leq i,j\leq N}\in M_N(\mc^\times)$.
\par
For an arbitrary matrix $A=(q_{ij})\in M_N(\mc^\times)$, we let  $\mathfrak{B}(V_A)$ denote the Nichols algebra associated to the $H$-Yetter-Drinfel'd module $V$ of diagonal type with the braiding matrix $A$. If the matrix $A$ under consideration is fixed, we denote it by $\mathfrak{B}(V)$ for short.
\par
From now on let $I=\{1,\cdots,N\}$ denote the index set.

\subsection{Differential operators}\label{section:diff}
Let $V\in{}_H^H\mathcal{YD}$ be an $H$-Yetter-Drinfel'd module of diagonal type and $\{v_1,\cdots,v_N\}$ be the basis of $V$ as fixed in the last subsection.

\begin{definition}[\cite{KRT94}]
Let $A$ and $B$ be two Hopf algebras with invertible antipodes. A
generalized Hopf pairing between $A$ and $B$ is a bilinear form
$\vp:A\times B\ra \mc$ such that:
\begin{enumerate}
\item For any $a\in A$, $b,b'\in B$,
$\vp(a,bb')=\sum\vp(a_{(1)},b)\vp(a_{(2)},b')$;
\item For any $a,a'\in A$, $b\in B$,
$\vp(aa',b)=\sum\vp(a,b_{(2)})\vp(a',b_{(1)})$;
\item For any $a\in A$, $b\in B$,
$\vp(a,1)=\ve(a),\ \ \vp(1,b)=\ve(b)$.
\end{enumerate}
\end{definition}

Let $\vp$ be a generalized Hopf pairing on $T(V)$ and itself such that $\vp(v_i,v_j)=\delta_{ij}$ (Kronecker delta notation). This pairing is not necessarily non-degenerate, whose radical is the defining ideal $\mathfrak{I}(V)$; it may pass to the quotient to give a non-degenerate generalized Hopf pairing on $\mathfrak{B}(V)$ (see, for example, Section 3.2 in \cite{AG99} for details).

\begin{definition}[\cite{MS}, Proposition 2.4; \cite{AHS}, Section 2.1; \cite{Fang11}, Definition 14]
The left and right differential operators associated to the element $a\in T(V)$ are defined by: 
\begin{enumerate}
\item $\p_a^L:T(V)\ra T(V)$, $\p_a^L(x)=\sum \vp(a,x_{(1)})x_{(2)}$;
\item $\p_a^R:T(V)\ra T(V)$, $\p_a^R(x)=\sum x_{(1)}\vp(a,x_{(2)})$.
\end{enumerate}
If $a=v_i$ for some $i\in I$, they will be denoted by $\p_i^L$ and $\p_i^R$, respectively.
\end{definition}

These differential operators descend to give endomorphisms of $\mathfrak{B}(V)$, which will also be denoted by $\p_a^L$ and $\p_a^R$.
\par
The following lemma, whose proof is trivial, will be useful. It holds when $T(V)$ is replaced by $\mathfrak{B}(V)$.
\begin{lemma}\label{Lemma:comm}
\begin{enumerate}
\item For any $a,x\in T(V)$, we have: 
$$\Delta(\p_a^L(x))=\sum\p_a^L(x_{(1)})\ts x_{(2)},\,\,\, \Delta(\p_a^R(x))=\sum x_{(1)}\ts \p_a^R(x_{(2)}).$$
\item For any $a,b\in T(V)$, $\p_a^L\p_b^R=\p_b^R\p_a^L$.
\end{enumerate}
\end{lemma}

\section{Identities in braid groups}\label{Sec5.3}
\subsection{Braid groups}
Let $n\geq 2$ be an integer, $\mathfrak{S}_n$ be the symmetric group: it acts on an alphabet of $n$ letters by permuting their positions and is generated by the set of transpositions $\{s_i=(i,i+1)|\ 1\leq i\leq n-1\}$. 
\par
Let $\mathfrak{B}_n$ be the braid group of $n$ strands. It is defined by generators $\s_i$ for $1\leq i\leq n-1$ and relations:
$$\s_i\s_j=\s_j\s_i,\ \ \text{for}\ |i-j|\geq 2;\ \ \s_i\s_{i+1}\s_i=\s_{i+1}\s_i\s_{i+1},\ \ \text{for}\ 1\leq i\leq n-2.$$ 
Let $\sigma=s_{i_1}\cdots s_{i_r}\in\mathfrak{S}_n$ be a reduced expression of $\sigma$. We denote the corresponding lifted element $T_\sigma=\s_{i_1}\cdots\s_{i_r}\in\mathfrak{B}_n$. This gives a linear map $T: \mc[\mathfrak{S}_n]\ra \mc[\mathfrak{B}_n]$ called the Matsumoto-Tits section.

\subsection{Defining ideals}
The total symmetrization operator in $\mc[\mathfrak{B}_n]$ is defined by:
$$S_n=\sum_{\sigma\in\mathfrak{S}_n}T_\sigma\in \mc[\fb_n].$$
Since $V\in{}^H_H\mathcal{YD}$ admits a braiding $\s$, $\fb_n$ acts naturally on $V^{\ts n}$ via $\s_i\mapsto {\id}^{\ts (i-1)}\ts\s\ts{\id}^{\ts(n-i-1)}$, which allows us to look $S_n$ as an endomorphism of $V^{\ts n}$.

\begin{proposition}[\cite{Ros98}, Proposition 9; \cite{Sbg96}, Corollary 2.4 and Theorem 2.7; \cite{AG99}, Proposition 3.2.12]\label{schauenburg}
Let $V$ be an $H$-Yetter-Drinfel'd module. Then
$$\mathfrak{B}(V)=\bigoplus_{n\geq 0} \left(V^{\ts n}/\ker(S_n)\right).$$
\end{proposition}

Details of this proposition and some different characterizations of the defining ideal can be found in \cite{AG99}.
\par
By this proposition, $\mathfrak{B}(V)$ can be viewed as imposing relations in $T(V)$ which are annihilated by the total symmetrization map, locating defining relations of $\mathfrak{B}(V)$ can be reduced to the study of each subspace $\ker(S_n)$. 

\subsection{Particular elements in braid groups and their relations}

We start by introducing some particular elements in the group algebra of braid groups. 
\begin{definition}
Let $n\geq 2$ be an integer. We define the following elements in $\mc[\mathfrak{B}_n]$:
$$\text{Garside element:       }\Delta_n=(\sigma_1\cdots\sigma_{n-1})(\sigma_1\cdots\sigma_{n-2})\cdots(\sigma_1\sigma_2)\sigma_1;$$
$$\text{Central element:       }\theta_n=\Delta_n^2;$$
$$\text{Right differential element:     }T_n=1+\s_{n-1}+\s_{n-1}\s_{n-2}+\cdots+\s_{n-1}\s_{n-2}\cdots\s_1;$$
$$\text{Right Dynkin element:       }P_{n}=(1-\s_{n-1}\s_{n-2}\cdots\s_1)(1-\s_{n-1}\s_{n-2}\cdots\s_{2})\cdots(1-\s_{n-1});$$
$$T_n'=(1-\s_{n-1}^2\s_{n-2}\cdots\s_1)(1-\s_{n-1}^2\s_{n-2}\cdots\s_2)\cdots(1-\s_{n-1}^2);$$
$$\text{Left differential element:     }U_n=1+\s_{1}+\s_{1}\s_{2}+\cdots+\s_{1}\s_{2}\cdots\s_{n-1};$$
$$\text{Left Dynkin element:       }Q_{n}=(1-\s_{1}\s_{2}\cdots\s_{n-1})(1-\s_{1}\s_{2}\cdots\s_{n-2})\cdots(1-\s_{1});$$
$$U_n'=(1-\s_{1}^2\s_{2}\cdots\s_{n-1})(1-\s_{1}^2\s_{2}\cdots\s_{n-2})\cdots(1-\s_{1}^2).$$
\end{definition}

We give a summary for some known results on the relations between them:

\begin{proposition}[\cite{Fang11}, \cite{KT08}]\label{braididentity}
The following identities hold:
\begin{enumerate}
\item For $n\geq 3$, $Z(\fb_n)$, the centre of $\fb_n$, is generated by $\theta_n$;
\item For any $1\leq i\leq n-1$, $\sigma_i\Delta_n=\Delta_n\s_{n-i}$;
\item $\theta_n=\Delta_n^2=(\s_{n-1}\s_{n-2}\cdots\s_1)^n=(\s_{n-1}^2\s_{n-2}\cdots\s_1)^{n-1}$;
\item $\left(\sum_{k=0}^{n-2}(\s_{n-1}^2\s_{n-2}\cdots\s_1)^k\right)(1-\s_{n-1}^2\s_{n-2}\cdots\s_1)=1-\Delta_n^2=1-\theta_n$;
\item $S_n=T_2T_3\cdots T_n=U_2U_3\cdots U_n$;
\item $T_nP_n=T_n'$,\, $U_nQ_n=U_n'$.
\end{enumerate}
\end{proposition}

\begin{proof}
(1) and (2) are proved in \cite{KT08}, Theorem 1.24. (3) is Lemma 4 and Proposition 4 in \cite{Fang11}. (4) is Corollary 2, \textit{ibid.} For the first identities of (5) and (6), see Proposition 5 and Proposition 6, \textit{ibid.} Applying $\Delta_n$ on these identities gives $S_n=U_2\cdots U_n$ and $U_nQ_n=U_n'$.
\end{proof}

We fix the notation for the embedding of braid groups at a fixed position.

\begin{definition}
For $m\geq 3$ and $2\leq k\leq m-1$, we let $\iota_{k}^m:\mathfrak{B}_k\ra\mathfrak{B}_m$ denote the group homomorphism defined by $\iota_k^m(\s_i)=\s_{m-k+i}$.
\end{definition}

\subsection{Relations with differential operators}
The following lemma explains the relation between the operator $T_n$ and the differential operators $\p_i^R$ defined in the Section \ref{section:diff}.

\begin{lemma}\label{Lem5.2}
Let $x\in T^n(V)$. The following statements are equivalent:
\begin{enumerate}
\item $T_nx=0$;
\item for any $i\in I$, $\p_i^R(x)=0$.
\end{enumerate}
\end{lemma}

\begin{proof}[Proof]
It comes from the following identity, which is clear from definition: for any $x\in T^n(V)$, 
\begin{equation}\label{**}
T_nx=\sum_{i\in I} \p_i^R(x)v_i.
\end{equation}
\end{proof}

\begin{remark}
The same result holds for left operators: $U_nx=0$ if and only if for any $i\in I$, $\p_i^L(x)=0$.
\end{remark}

\subsection{Tensor space representation of $\mathfrak{B}_n$}\label{tensorrep}
An $n$-uplet $\underline{i}=(i_1,\cdots,i_m)\in\mathbb{N}^m$ is called a partition of $n$, denote by $\underline{i}\vdash n$, if $i_1+\cdots+i_m=n$. Suppose from now on that $V\in{}^H_H\mathcal{YD}$ is of diagonal type with braiding matrix $(q_{ij})$. The braid group $\fb_n$ acts on $V^{\ts n}$, making it a $\mc[\fb_n]$-module. Since the braiding is of diagonal type, the $\mc[\fb_n]$-module $V^{\ts n}$ admits a decomposition into its submodules:
\begin{equation}\label{Eq:decomposition}
V^{\ts n}=\bigoplus_{\underline{i}\in \mathcal{P}_n}\mc[\fb_n].v_1^{i_1}\cdots v_m^{i_m},
\end{equation}
where $\mathcal{P}_n=\{\underline{i}=(i_1,\cdots,i_m)\in\mathbb{N}^m|\ \underline{i}\vdash n\}.$
\par
To simplify notations, for $\underline{i}=(i_1,\cdots,i_m)$, we denote $v_{\underline{i}}:=v_1^{i_1}\cdots v_m^{i_m}$ and the $\mc[\mathfrak{B}_n]$-module $\mc[X_{\underline{i}}]:=\mc[\fb_n].v_{\underline{i}}$.
\par
We let $\mathcal{H}$ denote the set of invariants $(V^{\ts n})^{\theta_n}$ under the action of the central element $\theta_n$. As $\theta_n\in Z(\fb_n)$, $\theta_nv_{\underline{i}}=v_{\underline{i}}$ implies that $\mc[X_{\underline{i}}]\subset\mathcal{H}$. Moreover, there exists some subset $J\subset \mathcal{P}_n$ such that 
$$\mathcal{H}=\bigoplus_{\underline{i}\in J}\mc[X_{\underline{i}}]$$ 
(see the argument in Section 6.1 of \cite{Fang11}).
\par
We finish this subsection by the following remark, which will frequently appear in the following discussions.

\begin{remark}\label{useful}
Suppose that $\underline{i}\vdash n$, $x\in \mc[X_{\underline{i}}]$ and $v\in V$. Then $(\id-\s_n\cdots\s_1)(vx)$ is in the ideal generated by $x$.
\par
To show this, notice that the coefficient $\lambda$ such that $\s_n\cdots\s_1(vx)=\lambda xv$ only depends on the chosen partition $\underline{i}$, it is therefore a constant for any $x\in \mc[X_{\underline{i}}]$.
\end{remark}

\subsection{Defining ideals of degree 2}

Elements of degree two in the defining ideal can be easily computed. They are characterized by the following proposition:

\begin{proposition}\label{Prop:deg2}
The following statements are equivalent: for $i\neq j$,
\begin{enumerate}
\item $q_{ij}q_{ji}=1$;
\item $v_iv_j-q_{ij}v_jv_i\in\ker S_2$;
\item $\theta_2(v_iv_j)=v_iv_j$.
\end{enumerate}
\end{proposition}

\begin{proof}[Proof]
It suffices to prove that (2) is equivalent to (3). Notice that $v_iv_j-q_{ij}v_jv_i=P_2(v_iv_j)$ and $S_2=T_2$. Then $T_2P_2(v_iv_j)=0$ if and only if $T_2'(v_iv_j)=0$ if and only if $\theta_2(v_iv_j)=v_iv_j$.
\end{proof}

\section{Another characterization of $\mathfrak{I}(V)$}\label{Sec5.4}

In this section, we give a characterization for a generating set of the defining ideal $\mathfrak{I}(V)$ using kernels of operators $T_n$, which is motivated by the work of Fr\o nsdal and Galindo \cite{FG00}. In fact, we could use the left or right differential operators to give a full characterization of the generators of the defining ideal.
\par
The following definition is due to Fr\o nsdal and Galindo \cite{FG00}:
\begin{definition}
An element $x\in T^n(V)$ is called a right (left) constant of degree $n$ if $T_nx=0$ $\CL U_nx=0\CR$. We let ${\Con}^R_n$ $\CL\Con^L_n\CR$ denote the vector space generated by all right (left) constants of degree $n$ and for any $m\geq 2$, we define
$${\Con}_{\leq m}^R=\Span\left(\bigcup_{2\leq n\leq m} {\Con}_n^R\right),\ \ {\Con}^R=\Span\left(\bigcup_{n\geq 2} {\Con}_n^R\right),$$
$${\Con}_{\leq m}^L=\Span\left(\bigcup_{2\leq n\leq m} {\Con}_n^L\right),\ \ {\Con}^L=\Span\left(\bigcup_{n\geq 2} {\Con}_n^L\right),$$
where the notation $\Span(X)$ stands for the $\mc$-vector space generated by the set $X$.
\end{definition}

The main technical tool is the following non-commutative version of the Taylor lemma for the diagonal braiding.
\begin{lemma}[Taylor Lemma, \cite{FG00}]
\begin{enumerate}
\item (Left version) For any integer $l\geq 1$ and $\underline{i}=(i_1,\cdots,i_l)\in\{1,\cdots,l\}^l$, there exists 
$$A^{\underline{i}}=\sum_{\s\in\mathfrak{S}_l}A^\s v_{i_{\s(1)}}\cdots v_{i_{\s(l)}}\in T^l(V)$$
with $A^\s\in \mc$ such that for any $x\in T^m(V)$,
$$x=c(x)+\sum_{l\geq 1}\sum_{\underline{i}\in\{1,\cdots,l\}^l}A^{\underline{i}}\p_{i_1}^L\cdots\p_{i_l}^L(x),$$
where $c(x)\in T^m(V)$ satisfies $\p_i^L(c(x))=0$ for any $i\in I$.

\item (Right version) For any integer $l\geq 1$ and $\underline{i}=(i_1,\cdots,i_l)\in\{1,\cdots,l\}^l$, there exists 
$$B^{\underline{i}}=\sum_{\s\in\mathfrak{S}_l}B^\s v_{i_{\s(1)}}\cdots v_{i_{\s(l)}}\in T^l(V)$$
with $B^\s\in \mc$ such that for any $x\in T^m(V)$,
$$x=d(x)+\sum_{l\geq 1}\sum_{\underline{i}\in\{1,\cdots,l\}^l}\p_{i_1}^R\cdots\p_{i_l}^R(x)B^{\underline{i}},$$
where $d(x)\in T^m(V)$ satisfies $\p_i^R(d(x))=0$ for any $i\in I$.
\end{enumerate}
\end{lemma}

\begin{lemma}
For any $m\geq 2$, ${\Con}_{\leq m}^L$ and ${\Con}_{\leq m}^R$ are coideals in the coalgebra $T^{\leq m}(V)$.
\end{lemma}

\begin{proof}[Proof]
We prove it for ${\Con}_{\leq m}^R$: taking $x\in\ker T_n$ for some $n\leq m$, then for any $i\in I$, $\p_i^R(x)=0$, which implies that (see Lemma \ref{Lemma:comm})
$$0=\Delta(\p_i^R(x))=\sum x_{(1)}\ts \p_i^R(x_{(2)})$$
and therefore for any $i\in I$, $\p_i^R(x_{(2)})=0$. This shows $\Delta(x)-x\ts 1\in T^{\leq m}(V)\ts {\Con}_{\leq m}^R$ and
$$\Delta(x)\in {\Con}_{\leq m}^R\ts T^{\leq m}(V)+T^{\leq m}(V)\ts {\Con}_{\leq m}^R.$$
\end{proof}

For a ring $R$ and a subset $X\subset R$, $<X>_{\ideal}$ denotes the ideal in $R$ generated by $X$.

\begin{theorem}\label{Thm:1}
For any $m\geq 2$, let 
$$R^m=\left<{\Con}_{\leq m}^R\right>_{\ideal}\, \cap\, T^m(V).$$
Then $R^m=\ker S_m$.
\end{theorem}

\begin{proof}[Proof]
Since $S_r=T_2T_3\cdots T_r$  and $\ker(T_r:T^r(V)\ra T^r(V))\subset\ker S_r\subset\mathfrak{I}(V)$, the inclusion $R^m\subset \ker S_m$ comes from the fact that $\mathfrak{I}(V)=\bigoplus_{m\geq 2}\ker S_m$ is an ideal. 
It suffices to show the other inclusion. Take $x\in\ker S_m$, we prove that $x\in R^m$ by induction on $m$. The case $m=2$ is clear as $T_2=S_2$. \\
\indent
Suppose that for any $2\leq k\leq m-1$, $R^k=\ker S_k$. It suffices to show that if for any $i\in I$, $\p_i^R(x)\in\ker S_{m-1}$, then $x\in R^m$. Indeed, for an element $x\in\ker S_m$, there are two cases:
\begin{enumerate}
\item $T_mx=0$; in this case, $x\in R^m$ is clear by definition.
\item $T_mx\neq 0$; from the decomposition of $S_n$, $T_mx\in\ker S_{m-1}$. According to (\ref{**}), it implies that for any $i\in I$, $\p_i^R(x)\in\ker S_{m-1}$. The proof will be terminated if the above statement is proved.
\end{enumerate}
We proceed to show the above statement. The following lemma is needed.

\begin{lemma}
For any $k\geq 3$, if $x\in R^k$, then $\p_i^R(x)\in R^{k-1}$ for any $i\in I$.
\end{lemma}

We continue the proof of the theorem. Let $x\in T^m(V)$ such that for any $i\in I$, $\p_i^R(x)\in\ker S_{m-1}=R^{m-1}$. From the right version of the Taylor lemma, 
$$x=d(x)+\sum_{l\geq 1}\sum_{\underline{i}\in\{1,\cdots,l\}^l}\p_{i_1}^R\cdots\p_{i_l}^R(x)B^{\underline{i}}.$$
The first term $d(x)$ in the right hand side satisfies $T_m(d(x))=0$ so is contained in $R^m$. Moreover, the hypothesis on $\p_i^R(x)$ and the lemma above force $\p_{i_1}^R\cdots\p_{i_n}^R(x)$ to be in $R^{m-n}$, so the second term is in $R^m$.
\end{proof}

Now it remains to prove the lemma.

\begin{proof}[Proof of the lemma]
It suffices to deal with the case where $x=urw\in R^k$ such that $r\in \ker T_s\cap \mc[X_{\underline{i}}]$ for some $\underline{i}\vdash s$, $u\in T^p(V)$ and $w\in T^q(V)$ satisfying $k=s+p+q$.\\
\indent
We have the following decomposition: $T_k=T_k^1+T_k^2+T_k^3$ where
$$T_k^1=1+\s_{k-1}+\s_{k-1}\s_{k-2}+\cdots +\s_{k-1}\cdots\s_{p+s+1},$$
$$T_k^2=\s_{k-1}\cdots\s_{p+s}(\iota_{s}^{p+s}(T_s)),$$
$$T_k^3=\s_{k-1}\cdots\s_{p-1}+\cdots+\s_{k-1}\cdots\s_1.$$
It is clear that $T_k^2x=0$. By Remark \ref{useful}, both $T_k^1x$ and $T_k^3x$ are contained in $R^k$ since they are in the ideal generated by $r$. Moreover, it should be remarked that in $T_kx$, $r$ is always contained in the first $k-1$ tensorands. This is true by definitions of $T_k^1$ and $T_k^3$.\\
\indent
As a conclusion, we have shown that $T_kx\in R^k$, so $\p_i^R(x)\in R^{k-1}$ for any $i\in I$, by the formula (\ref{**}).
\end{proof}

In the proof of the above theorem, we have shown as a byproduct the following proposition, which can be looked as a kind of "invariance under integration".

\begin{proposition}\label{Cor:diff}
For $x\in T^m(V)$ where $m\geq 3$, the following statements are equivalent:
\begin{enumerate}
\item $x\in R^m$;
\item for any $i\in I$, $\p_i^R(x)\in R^{m-1}$.
\end{enumerate}
\end{proposition}

These results are correct when the prefix "right" is replaced by "left". The proof above can be adapted by using the left version of the Taylor lemma. We omit these statements while give the following corollary.

\begin{corollary}
Let 
$$L^m=\left<{\Con}_{\leq m}^L\right>_{\ideal}\, \cap\, T^m(V).$$
Then $R^m=L^m=\ker S_m$.
\end{corollary}

As a conclusion, to find the generating relations, it suffices to consider those in the intersection of $\ker\p_i^R$ for all $i\in I$, or the intersection of $\ker\p_i^L$ for all $i\in I$. We will see in the next section a refined result giving more constraints.

\begin{remark}\label{leftright}
Globally, when passing to the generating ideal, there is no difference between the left and right cases. But an element annihilated by all right differentials may not necessarily contained in the kernel of all $\p_i^L$. We will return to this problem in Section \ref{Section:Balance}.
\end{remark}

The following lemma is a direct consequence of Lemma \ref{Lemma:comm}.
\begin{lemma}
For any $i\in I$ and $m\geq 3$, $\p_i^R$ $\CL$resp. $\p_i^L\CR$ sends ${\Con}_{\leq m}^L$ $\CL$resp. ${\Con}_{\leq m}^R\CR$ to ${\Con}_{\leq m-1}^L$ $\CL$resp. ${\Con}_{\leq m-1}^R\CR$.
\end{lemma}

\section{Defining relations in the diagonal type}\label{Sec5.5}

\subsection{More constraints: pre-relations}
In this subsection, we propose a smaller set of generators in $\mathfrak{I}(V)$ by posing more constraints to the left and right constants. These constraints give a restriction on the degrees where this new set of generators may appear. We start by some motivations towards the main definition.
\par
Let $H$ be a Hopf algebra and $X\subset H$ be a subset. The Hopf ideal generated by $X$ is the smallest Hopf ideal containing $X$.

\begin{proposition}[\cite{Fang11}]
The Hopf ideal in ${}^{H}_H\mathcal{YD}$ generated by $\bigoplus_{n\geq 2}\left(\ker(S_n)\cap \im(P_n)\right)$ is $\mathfrak{I}(V)$.
\end{proposition}

This proposition, combined with Theorem \ref{Thm:1}, gives more constraints.

\begin{corollary}\label{Cor:2}
The Hopf ideal in ${}^{H}_H\mathcal{YD}$ generated by $\bigoplus_{n\geq 2}\left(\ker(T_n)\cap \im(P_n)\right)$ is $\mathfrak{I}(V)$.
\end{corollary}

Thanks to this corollary, to find relations imposed in $\mathfrak{B}(V)$, it suffices to concentrate on elements in $\im(P_n)$ annihilated by all right differentials $\p_i^R$.
\par
According to Proposition \ref{braididentity} (6), to find a generating set of $\mathfrak{I}(V)$, it suffices to consider the solution of the equation $T_n'x=T_nP_nx=0$ in $T^n(V)$. This observation motivates the following definition:

\begin{definition}
Let $n\geq 2$ be an integer. We call a non-zero element $v\in T^n(V)$ a right pre-relation of degree $n$ if
\begin{enumerate}
\item $T_nv=0$ and $v=P_nw$ for some $w\in T^n(V)$;
\item $\iota_{n-1}^n(T_{n-1}')w\neq 0$.
\end{enumerate}
Let ${\Rel}_r^n$ denote the vector space generated by all right pre-relations of degree $n$ and ${\Rel}_r$ denote the vector space generated by $\bigcup_{n\geq 2}{\Rel}_r^n$. Elements in ${\Rel}_r$ are called right pre-relations.
\end{definition}

We can similarly define left pre-relations of degree $n$ by replacing $T_n$ by $U_n$, $T_{n-1}'$ by $U_{n-1}'$ and $P_n$ by $Q_n$ in the above definition. Let ${\Rel}_l^n$ denote the $\mc$-vector space generated by all left pre-relations of degree $n$ and  ${\Rel}_l$ denote the $\mc$-vector space generated by $\bigcup_{n\geq 2}{\Rel}_l^n$. Elements in ${\Rel}_l$ are called left pre-relations.

\begin{remark}
They are called pre-relations as they may be redundant.
\end{remark}

We establish some properties of the pre-relations. Recall the definition of $T_n'$:
$$T_n'=(1-\s_{n-1}^2\s_{n-2}\cdots\s_1)(1-\s_{n-1}^2\s_{n-2}\cdots\s_2)\cdots(1-\s_{n-1}^2\s_{n-2})(1-\s_{n-1}^2).$$
We define the following elements in $\mc[\mathfrak{B}_n]$ for $1\leq m\leq n-1$:
$$X_{m,n}=(1-\s_{n-1}^2\s_{n-2}\cdots\s_{n-m})\cdots(1-\s_{n-1}^2\s_{n-2})(1-\s_{n-1}^2)=\iota_{m+1}^n(T_{m+1}'),$$
then $X_{1,n}=(1-\s_{n-1}^2)$, $X_{n-1,n}=T_n'$ and $X_{n-2,n}=\iota_{n-1}^n(T_{n-1}')$.

\begin{proposition}\label{Prop:theta}
If $T_n'w=0$ and $X_{n-2,n}w\neq 0$ for some $w\in T^n(V)$, then $\theta_nw=w$.
\end{proposition}

\begin{proof}[Proof]
From definition, $T_n'w=(1-\s_{n-1}^2\s_{n-2}\cdots\s_1)X_{n-2,n}w$. If $X_{n-2,n}w\neq 0$, it will be a solution of the equation $(1-\s_{n-1}^2\s_{n-2}\cdots\s_1)x=0$. By Proposition \ref{braididentity}, multiplying both sides by $\sum_{k=0}^{n-2}(\s_{n-1}^2\s_{n-2}\cdots\s_1)^k$ gives $\theta_nX_{n-2,n}w=X_{n-2,n}w$. This implies $\theta_nw=w$ by applying the argument at the beginning of Section \ref{tensorrep}.
\end{proof}

\begin{corollary}\label{Cor:theta}
Let $v\in T^n(V)$ be a right pre-relation. Then $\theta_nv=v$.
\end{corollary}

\begin{proof}
By Proposition \ref{Prop:theta}, if $v=P_nw$ for some $w\in T^n(V)$, then $\theta_nw=w$. Therefore 
$$\theta_n v=\theta_nP_nw=P_n\theta_nw=P_nw=v.$$
\end{proof}

As a conclusion, to solve the equation $T_nx=0$ in aim of finding defining relations, it suffices to work inside the $\mc[\mathfrak{B}_n]$-module $\mc[X_{\underline{i}}]$ such that $\theta_n (v_{\underline{i}})=v_{\underline{i}}$.

\begin{corollary}
If $w\in T^n(V)$ is such that $T_n'w=0$ and for any $2\leq k\leq n-1$, $\iota_k^n(\theta_k)w\neq w$, then $\theta_nw=w$.
\end{corollary}

\begin{proof}[Proof]
By Proposition \ref{Prop:theta}, it suffices to show that $X_{n-2,n}w\neq 0$. Otherwise, take the smallest $k$ such that $X_{k-1,n}w\neq 0$ but $X_{k,n}w=0$. As $X_{k,n}=\iota_{k+1}^n(T_{k+1}')$, Proposition \ref{Prop:theta} can be applied to this case to give $\iota_{k+1}^n(\theta_{k+1})w=w$. This contradicts the hypothesis.
\end{proof}

The main result of this section is the following theorem.

\begin{theorem}\label{Relr}
The Hopf ideal generated by ${\Rel}_r$ is $\mathfrak{I}(V)$. 
\end{theorem}

\begin{proof}
Let $w\in T^n(V)$ be a solution of the equation $T_n'x=X_{n-1,n}x=0$. There are two possibilities: 
\begin{enumerate}
\item $X_{n-2,n}w\neq 0$; it is clear that $P_nw$ is a right pre-relation.
\item $X_{n-2,n}w=0$; then there exists a smallest $k$ such that $X_{k-1,n}v\neq 0$ but $X_{k,n}v=0$.
\end{enumerate}

We would like to show that only relations fallen into the first case are interesting. To be more precise, if $w$ falls into the second case, $P_nw$ can be generated by lower degree elements in the first case. This is the following lemma.

\begin{lemma}\label{Lem:7}
If $w\in T^n(V)$ is an element such that $T_n'w=0$ and $X_{n-2,n}w=0$, then $P_nw$ is in the ideal generated by right pre-relations of lower degrees.
\end{lemma}

\begin{proof}[Proof]
The proof is executed by induction on $n$. There is nothing to prove for $n=2$.
\par
Let $w\in T^n(V)$ be such that $T_n'w=0$ and $X_{n-2,n}w=0$ then set $k$ be the smallest integer such that $X_{k-1,n}w\neq 0$ but $X_{k,n}w=0$. By definition of $P_n$, 
$$P_nw=(1-\s_{n-1}\cdots\s_1)\cdots (1-\s_{n-1}\cdots\s_{k+1})\iota_{k+1}^n(P_{k+1})w.$$
We write 
$$w=\sum_i\sum_{\underline{j}\vdash k+1}u_i\ts w_{i,\underline{j}},$$
where $u_i\in T^{n-k-1}(V)$ are linearly independent and $w_{i,\underline{j}}\in T^{k+1}(V)\cap \mc[X_{\underline{j}}]$. Recall that $X_{k,n}=\iota_{k+1}^n(T_{k+1}')$, then $X_{k,n}w=0$ implies that 
$$\sum_{\underline{j}\vdash k+1}T_{k+1}'w_{i,\underline{j}}=0.$$
As these $\mc[X_{\underline{j}}]$ have trivial intersection, for any $\underline{j}$, $X_{k,n}w_{i,\underline{j}}=0$.
\par
There are two cases:
\begin{enumerate}
\item $T_k'w_{i,\underline{j}}=0$. In this case, by applying the induction hypothesis to $w_{i,\underline{j}}$, $P_kw_{i,\underline{j}}$ is generated by right pre-relations of lower degrees. So $P_n(u_iw_{i,\underline{j}})$ is generated by right pre-relations of lower degrees by Remark \ref{useful}.
\item $T_k'w_{i,\underline{j}}\neq 0$, then $P_kw_{i,\underline{j}}$ is a right pre-relation of degree $k$ and $P_n(u_iw_{i,\underline{j}})$ is generated by right pre-relations of lower degree by Remark \ref{useful}.
\end{enumerate}
As a summary, for any $i$ and $\underline{j}$, $P_n(u_iw_{i,\underline{j}})$ is generated by right pre-relations lower degree, so is $P_nw$.
\end{proof}
By Corollary \ref{Cor:2}, to terminate the proof of the theorem, it suffices to show that the Hopf ideal generated by $\bigoplus_{n\geq 2}\left(\ker(T_n)\cap \im(P_n)\right)$ coincides with the Hopf ideal generated by $\Rel_r$. 
\par
Take $x\in\ker(T_n)\cap\im(P_n)$, there exists $w$ such that $P_nw=x$ and $T_n'w=P_nT_nx=0$. By the above argument, if $X_{n-2,n}w\neq 0$, then by definition $x\in\Rel_r$; if not, by Lemma \ref{Lem:7}, $x=P_nw$ is contained in the Hopf ideal generated by $\Rel_r$.
\end{proof}

\begin{example}\label{Ex:deg2}
We compute pre-relations of degree $2$. Since $P_2=Q_2$ and $T_2=U_2$, ${\Rel}_r^2$ coincides with ${\Rel}_l^2$. It suffices to consider each $\mc[X_{\underline{i}}]$ where $\underline{i}=(s,t)$. The following facts are clear by Proposition \ref{Prop:deg2}:
\begin{enumerate}
\item $T_2P_2=1-\theta_2$ acts as zero on ${\Rel}_r^2$, so it suffices to consider the fixed points of $\theta_2$;
\item $\theta_2 v_{\underline{i}}=v_{\underline{i}}$ if and only if $q_{st}q_{ts}=1$.
\end{enumerate}
These observations give the following characterization of ${\Rel}_r^2$:
$${\Rel}_r^2=\Span\{v_sv_t-q_{st}v_tv_s|\ s< t\ \textrm{such that }q_{ts}q_{st}=1\ \text{and}\ s=t,\ q_{ss}=-1\}.$$
There is no redundant relations in this list and it coincides with the set of constants of degree $2$.
\end{example}

\subsection{Balancing left and right}\label{Section:Balance}

The set of left and right constants or pre-relations may not coincide, we study in this subsection the symmetries between them. 
\par
The following lemma is clear by Proposition \ref{braididentity}.

\begin{lemma}
For any $n\geq 2$, $\Delta_nT_n=U_n\Delta_n$ and $\Delta_nP_n=Q_n\Delta_n$.
\end{lemma}

The Garside element gives the symmetry between the left and right pre-relations.

\begin{corollary}
The Garside element $\Delta_n$ induces a linear isomorphism ${\Rel}_r^n\cong {\Rel}_l^n$.
\end{corollary}

\begin{proof}[Proof]
According to Proposition \ref{Prop:theta}, $\Delta_n^2=\theta_n$ acts as the identity on ${\Rel}_r^n$, thus $\Delta_n$ is a linear isomorphism. It suffices to show that the image of $\Delta_n$ is contained in ${\Rel}_l^n$.
\par
We verify that $\Delta_nw\in \textrm{Rel}_l^n$ for $w\in {\Rel}_r^n$. The first condition holds by the above lemma and the other point comes from the injectivity of $\Delta_n$. If we write $w=P_nv$, then apply the above lemma again, $\Delta_nw=\Delta_nP_nv=Q_n\Delta_n v$ implies $\Delta_n w$ is in the image of $Q_n$. 
\end{proof}

A similar result holds when the pre-relations are replaced by constants.

\begin{corollary}
The Garside element $\Delta_n$ induces a linear isomorphism ${\Con}_r^n\cong {\Con}_l^n$.
\end{corollary}

\begin{proof}[Proof]
It is clear that $\Delta_n$ sends ${\Con}_r^n$ to ${\Con}_l^n$, then it suffices to show that $\Delta_n$ is an isomorphism.
\par
Thanks to the decomposition (\ref{Eq:decomposition}) and notations therein, we can decompose ${\Con}_r^n$ and ${\Con}_l^n$ into direct sums of the $\mc[\mathfrak{B}_n]$-modules $\mc[X_{\underline{i}}]$ for $\underline{i}\vdash n$ such that the action of $\theta_n$ is given by an invertible scalar on each summand space. So $\Delta_n$ induces a linear isomorphism
$$\mc[X_{\underline{i}}]\cap {\Con}_r^n\cong \mc[X_{\underline{i}}]\cap {\Con}_l^n$$
and therefore a linear isomorphism between ${\Con}_r^n$ and ${\Con}_l^n$.
\end{proof}

\section{Generalized quantum groups}\label{Sec5.6}

\subsection{Generalized quantum groups}
For a Nichols algebra $\mathfrak{B}(V)$ associated to a Yetter-Drinfeld module $V\in{}^H_H\mathcal{YD}$, the bosonization $\mathfrak{B}(V)\# H$ is a true Hopf algebra \cite{AS02}. This construction, once applied to the Nichols algebra of diagonal type associated to the data of a symmetrizable Kac-Moody Lie algebra, gives the positive or negative part of the corresponding quantum group. But here, we would like to define them in a more direct way.
\par
Let $\mc(q)$, the field of rational functions in one variable over $\mc$, be the base field in this subsection.

\begin{definition}
Let $A=(q_{ij})_{1\leq i,j\leq N}$ be a braiding matrix in $M_N(\mc(q))$ such that $q_{ij}=q^{n_{ij}}$ for some $n_{ij}\in\mathbb{Z}$. 
\begin{enumerate}
\item $T^{\leq 0}(A)$ is defined as the Hopf algebra generated by $F_i$, $K_i^{\pm 1}$ for $i\in I$ with relations:
$$K_iF_jK_i^{-1}=q_{ij}^{-1}F_j,\,\,\, K_iK_i^{-1}=K_i^{-1}K_i=1,$$
$$\Delta(F_i)=K_i\ts F_i+F_i\ts 1,\ \ \Delta(K_i^{\pm 1})=K_i^{\pm 1}\ts K_i^{\pm 1},\ \ \ve(F_i)=0,\ \ \ve(K_i)=1;$$
\item $T^{\geq 0}(A)$ is defined as the Hopf algebra generated by $E_i$, $K_i'^{\pm 1}$ for $i\in I$ with relations:
$$K_i'E_jK_i'^{-1}=q_{ij}E_j,\,\,\, K_i'K_i'^{-1}=K_i'^{-1}K_i'=1,$$
$$\Delta(E_i)=1\ts E_i+E_i\ts K_i'^{-1},\ \ \Delta(K_i'^{\pm 1})=K_i'^{\pm 1}\ts K_i'^{\pm 1},\ \ \ve(E_i)=0,\ \ \ve(K_i')=1;$$
\item $D^{\leq 0}(A)$ (resp. $D^{\geq 0}(A)$) is defined as the quotient of $T^{\leq 0}(A)$ (resp. $T^{\geq 0}(A)$) by the biideal generated by the right (resp. left) pre-relations.
\end{enumerate}
\end{definition}
We define a generalized Hopf pairing $\vp:T^{\geq 0}(A)\times T^{\leq 0}(A)\ra \mc(q)$ such that for any $i,j\in I$,
$$\vp(E_i,F_j)=-\frac{\delta_{ij}}{q-q^{-1}},\,\,\, \vp(K_i',K_j)=q_{ij},$$
$$\vp(E_i,K_j^{\pm 1})=\vp(K_i'^{\pm 1},F_j)=0.$$
By Theorem \ref{Relr}, pre-relations generate the defining ideal. It is shown in \cite{AG99}, Theorem 3.2.29 that radicals of the generalized Hopf pairing coincide with the defining ideal in $T^{\geq 0}(A)$ and $T^{\leq 0}(A)$, hence $\vp$ induces a non-degenerate generalized Hopf pairing $\vp:D^{\geq 0}(A)\times D^{\leq 0}(A)\ra \mc(q)$.
\par
The following quantum double construction allows us to define the generalized quantum group.

\begin{definition}[\cite{KRT94}, Theorem 3.2]
Let $A$, $B$ be two Hopf algebras with invertible antipodes and $\vp$ be
a generalized Hopf pairing between them. Their quantum double $D_\vp(A,B)$ is defined by:
\begin{enumerate}
\item as a vector space, it is $A\ts B$;
\item as a coalgebra, it is the tensor product of coalgebras $A$ and $B$;
\item as an algebra, the multiplication is given by: for $a,a'\in A$ and $b,b'\in B$,
$$(a\ts b)(a'\ts b')=\sum \vp(S^{-1}(a_{(1)}'),b_{(1)})\vp(a_{(3)}',b_{(3)})aa_{(2)}'\ts b_{(2)}b'.$$
\end{enumerate}
\end{definition}

\begin{definition}
Suppose moreover that $A$ is a symmetric matrix. The generalized quantum group $D_q(A)$ associated to the braiding matrix $A$ is defined by:
$$D_q(A)=D_\vp(D^{\geq 0}(A),D^{\leq 0}(A))/(K_i-K_i'|\,i\in I),$$
where $(K_i-K_i'|\,i\in I)$ is the Hopf ideal generated by $K_i-K_i'$ for $i\in I$.
\end{definition}

We can similarly define the Hopf algebra $T_q(A)$ by replacing $D^{\geq 0}(A)$ and $D^{\leq 0}(A)$ by $T^{\geq 0}(A)$ and $T^{\leq 0}(A)$. Then $D_q(A)$ is the quotient of $T_q(A)$ by the Hopf ideal generated by the defining ideals.
\par
A routine computation gives the commutation relation between $E_i$ and $F_j$:
$$[E_i,F_j]=\delta_{ij}\frac{K_i-K_i^{-1}}{q-q^{-1}}.$$

\begin{remark}
We use the notation $D_q(A)$ instead $U_q(A)$ as it may not be related to the universal enveloping algebra associated to a Kac-Moody Lie algebra. This phenomenon will be examined in Example \ref{Ex:main}.
\end{remark}

\subsection{Averaged quantum group}

We will be interested in a particular case of the above construction where the braiding matrix arises from a generalized Cartan matrix.
\par
Let $C=(c_{ij})_{1\leq i,j\leq N}$ be a generalized Cartan matrix in $M_N(\mathbb{Z})$, i.e., a matrix of integral entries satisfying: 
\begin{enumerate}
\item $c_{ii}=2$;
\item for any $i\neq j$, $c_{ij}\leq 0$;
\item $c_{ij}=0$ implies $c_{ji}=0$.
\end{enumerate}
In the following discussion, we take $K=\mc(q^{\frac{1}{2}})$ as the ground field since elements in our matrices may be in the additive group $\displaystyle\frac{1}{2}\mathbb{Z}$.

\begin{definition}
Let $C\in M_N(\mathbb{Z})$ be a generalized Cartan matrix.
\begin{enumerate}
\item The averaged matrix associated to $C$ is defined by $\overline{C}=(\overline{c}_{ij})_{1\leq i,j\leq N}\in M_N(\mathbb{Q})$ where $\overline{c}_{ij}=\frac{1}{2}(c_{ij}+c_{ji})$. We denote $A=(q^{\overline{c}_{ij}})_{1\leq i,j\leq N}$: it is a symmetric matrix.
\item The $q$-enveloping algebra $U_q(C)$ associated to $C$ is defined by: 
\begin{enumerate}
\item if $C$ is symmetrizable, we respect the original definition of the quantized enveloping algebra associated to the Kac-Moody Lie algebra $\g(C)$ as a $\mc(q^{\frac{1}{2}})$-algebra;
\item if $C$ is non-symmetrizable, it is $D_q(A)$ as a $\mc(q^{\frac{1}{2}})$-algebra.
\end{enumerate}
\end{enumerate}
\end{definition}

We let $D^{>0}(A)$ (resp. $D^{<0}(A)$) denote the subalgebra of $D^{\geq 0}(A)$ (resp. $D^{\leq 0}(A)$) generated by $E_i$ (resp. $F_i$) for $i\in I$. They are Nichols algebras associated to the braiding matrix $A=(q^{\overline{c}_{ij}})$ (resp. $A'=(q^{-\overline{c}_{ij}})$). The following result due to Andruskiewitsch and Schneider shows that passing to the averaged matrix will not lose too much information.

\begin{proposition}[\cite{AS02}]\label{Prop:AS}
Let $V$ and $V'$ be two Yetter-Drinfel'd modules of diagonal type with braiding matrices $(q_{ij})_{1\leq i,j\leq N}$ and $(q_{ij}')_{1\leq i,j\leq N}$ satisfying $q_{ij}q_{ji}=q_{ij}'q_{ji}'$ for any $i,j\in I$ with respect to bases $v_1,\cdots,v_N$ of $V$ and $v_1',\cdots,v_N'$ of $V'$. Then
\begin{enumerate}
\item there exists a linear isomorphism $\psi:\mathfrak{B}(V)\ra \mathfrak{B}(V')$ such that for any $i\in I$, $\psi(v_i)=v_i'$;
\item this linear map $\psi$ almost preserves the algebra structure: for any $i,j\in I$,
$$\psi(v_iv_j)=\left\{\begin{matrix}q_{ij}'q_{ij}^{-1}v_i'v_j' & \textrm{if}\ \ i\leq j ;\\ 
v_i'v_j', & \textrm{if}\ \ i>j.\end{matrix}\right.
$$
\end{enumerate}
\end{proposition}

\begin{remark}
We let $C$ be a non-symmetrizable generalized Cartan matrix. By Remark 1 and Theorem 21 in \cite{Ros98}, in order that $D^{<0}(A)$ to be of finite Gelfan'd-Kirillov dimension, the matrix $\overline{C}$ must be in $M_N(\mathbb{Z})$. This implies that a large number of algebras we are considering are of infinite Gelfand-Kirillov dimensions.
\end{remark}

\subsection{Bar involution in symmetric case}\label{Sec:Bar}

In this subsection, we suppose moreover that the braiding matrix is symmetric: that is to say, $q_{ij}=q_{ji}$ for any $i,j\in I$. This is always the case when the $q$-enveloping algebra $U_q(C)$ is under consideration.
\par
This hypothesis allows us to define the bar involution on Nichols algebras, which is fundamental in the study of quantum groups, especially for canonical (global crystal) bases.

\begin{definition}
The bar involution $-:T(V)\ra T(V)$ is a $\mc$-linear automorphism defined by $q^{\frac{1}{2}}\mapsto q^{-\frac{1}{2}}$ and $v_i\mapsto v_i$.
\end{definition}

\begin{definition}
For any $i\in I$, we define $\dd_i^R$, $\dd_i^L\in {\End}_{K}(T(V))$ to be the $\mc(q^{\frac{1}{2}})$-linear maps satisfying that for any monomial $v_{i_1}\cdots v_{i_l}$ in $T(V)$,
$${\dd}_i^R(v_{i_1}\cdots v_{i_l})=\overline{\p_i^R(v_{i_1}\cdots v_{i_l})},\ \ {\dd}_i^L(v_{i_1}\cdots v_{i_l})=\overline{\p_i^L(v_{i_1}\cdots v_{i_l})}$$
\end{definition}

We start by showing that bar involution descends to Nichols algebras. The following lemma will be needed.

\begin{lemma}\label{Lem:comm}
For any $i,j\in I$, $\p_j^R{\dd}_i^R=q_{ij}^{-1}{\dd}_i^R\p_j^R$.
\end{lemma}

\begin{proof}[Proof]
It is easy to show the following formulas: for any $i_1,\cdots,i_n\in I$,
$$\p_j^R(v_{i_1}\cdots v_{i_n})=q_{j,i_n}\p_j^R(v_{i_1}\cdots v_{i_{n-1}})v_{i_n}+v_{i_1}\cdots v_{i_{n-1}}\p_j^R(v_{i_n}),$$
$${\dd}_i^R(v_{i_1}\cdots v_{i_n})=q_{i,i_n}^{-1}{\dd}_i^R(v_{i_1}\cdots v_{i_{n-1}})v_{i_n}+v_{i_1}\cdots v_{i_{n-1}}{\dd}_i^R(v_{i_n}).$$
Since both $\p_i^R$ and ${\dd}_i^R$ are $\mc(q)$-linear, it suffices to verify the lemma for monomials. We use induction on the degree $n$ of the monomial. The case $n=1$ is trivial. Taking a monomial $v_{i_1}\cdots v_{i_n}$ and applying formulas above gives:\begin{eqnarray*}
& &\p_j^R{\dd}_i^R(v_{i_1}\cdots v_{i_n})\\
&=& q_{j,i_n}q_{i,i_n}^{-1}\p_j^R{\dd}_i^R(v_{i_1}\cdots v_{i_{n-1}})v_{i_n}+q_{i,i_n}^{-1}{\dd}_i^R(v_{i_1}\cdots v_{i_{n-1}})\p_j^R(v_{i_n})+\p_j^R(v_{i_1}\cdots v_{i_{n-1}}){\dd}_i^R(v_{i_n}),
\end{eqnarray*}
\begin{eqnarray*}
& &{\dd}_i^R\p_j^R(v_{i_1}\cdots v_{i_n})\\
&=& q_{j,i_n}q_{i,i_n}^{-1}{\dd}_i^R\p_j^R(v_{i_1}\cdots v_{i_{n-1}})v_{i_n}+q_{j,i_n}\p_j^R(v_{i_1}\cdots v_{i_{n-1}}){\dd}_i^R(v_{i_n})+{\dd}_i^R(v_{i_1}\cdots v_{i_{n-1}})\p_j^R(v_{i_n}).
\end{eqnarray*}
Then induction hypothesis can be applied to give
\begin{eqnarray*}
& &(\p_j^R{\dd}_i^R-q_{ij}^{-1}{\dd}_i^R\p_j^R)(v_{i_1}\cdots v_{i_n})\\
&=& (1-q_{ij}^{-1}q_{j,i_n}) \p_j^R(v_{i_1}\cdots v_{i_{n-1}}){\dd}_i^R(v_{i_n})+(q_{i,i_n}^{-1}-q_{ij}^{-1}){\dd}_i^R(v_{i_1}\cdots v_{i_{n-1}})\p_j^R(v_{i_n}).
\end{eqnarray*}

Notice that if $i\neq i_n$ and $j\neq i_n$, the right hand side is zero. It remains to tackle the following three cases:
\begin{enumerate}
\item if $i=j=i_n$, then the two coefficients on the right hand side are zero;
\item if $i\neq j$, $i=i_n$, then the coefficient of the second term in the right hand side vanishes and the first term is zero;
\item if $i\neq j$, $j=i_n$ then the coefficient of the first term in the right hand side vanishes and the second term is zero.
\end{enumerate}
\par
This terminates the proof.
\end{proof}

\begin{proposition}
The restriction of the bar involution induces a linear automorphism of $\mathfrak{I}(V)$.
\end{proposition}

\begin{proof}[Proof]
We let $\mathfrak{I}(V)_n$ denote the set of degree $n$ elements in $\mathfrak{I}(V)$. Since $-$ is an involution, it suffices to show $\overline{\mathfrak{I}(V)}\subset \mathfrak{I}(V)$. 
\begin{lemma}
For any $i\in I$, ${\dd}_i^R(\mathfrak{I}(V))\subset \mathfrak{I}(V)$.
\end{lemma}
\begin{proof}
The proof is executed by induction on the degree of elements in $\mathfrak{I}(V)$.
\par
The case $n=2$ is clear since by Proposition \ref{Prop:deg2}, $\mathfrak{I}(V)_2$ is generated as a vector space by $v_iv_j-q_{ij}v_jv_i$ for some $q_{ij}$ satisfying $q_{ij}q_{ji}=1$; the hypothesis on the braiding matrix forces $q_{ij}=\pm 1$, hence
$${\dd}_k^R(v_iv_j-q_{ij}v_jv_i)=\overline{\p_k^R(v_iv_j-q_{ij}v_jv_i)}$$
is zero since $\p_k^R$ annihilates $\mathfrak{I}(V)_2$.
\par
Take $v\in\mathfrak{I}(V)_n$, by Proposition \ref{Cor:diff}, it suffices to show that for any $j\in I$, $\p_j^R{\dd}_i^R(v)\in \mathfrak{I}(V)$. Applying Lemma \ref{Lem:comm} gives 
$$\p_j^R{\dd}_i^R(v)=q_{ij}^{-1}{\dd}_i^R\p_j^R(v),$$
by Proposition \ref{Cor:diff} again, $\p_j^R(v)\in \mathfrak{I}(V)$ with a lower degree, hence ${\dd}_i^R\p_j^R(v)\in \mathfrak{I}(V)$ by induction hypothesis and therefore $\p_j^R{\dd}_i^R(v)\in\mathfrak{I}(V)$.
\end{proof}

Return to the proof of the proposition: for $v\in\mathfrak{I}(V)_n$, we show that
$$\overline{\p_i^R(\overline{v})}={\dd}_i^R(v).$$
Indeed, we write $v=\sum \alpha_{i_1,\cdots,i_n}v_{i_1}\cdots v_{i_n}$ for some $\alpha_{i_1,\cdots,i_n}\in\mc(q^{\frac{1}{2}})$ where the sum is over $i_1,\cdots,i_n\in I$,
$$\overline{\p_i^R(\overline{v})}=\overline{\sum \overline{\alpha_{i_1,\cdots,i_n}}\p_i^R(v_{i_1}\cdots v_{i_n})}=\sum \alpha_{i_1,\cdots,i_n}\overline{\p_i^R(v_{i_1}\cdots v_{i_n})}={\dd}_i^R(v).$$

The proof is given by induction on the degree of $v\in \mathfrak{I}(V)_n$ to show that for any $i\in I$, $\p_i^R(\overline{v})\in\mathfrak{I}(V)$, then apply Proposition \ref{Cor:diff}.
\par
The case $n=2$ has shown between the lines of the above lemma. For $n\geq 3$, to show that $\p_i^R(\overline{v})\in\mathfrak{I}(V)$, it suffices to verify that $\overline{\p_i^R(\overline{v})}\in\mathfrak{I}(V)$ by the induction hypothesis and the fact that the bar map is an involution. But we have shown $\overline{\p_i^R(\overline{v})}={\dd}_i^R(v)$, which is in $\mathfrak{I}(V)$ by the above lemma. This finishes the proof.
\end{proof}

According to this proposition, the bar involution may pass the quotient to give a $\mc$-linear automorphism of the Nichols algebra $\mathfrak{B}(V)$.
\par
The relation between the bar involution and the action of the Garside element on the image of the Dynkin operator $P_n$ is explained in the following proposition.

\begin{proposition}
For any $v\in T^n(V)$ satisfying $\theta_nv=v$, $\Delta_nP_nv=(-1)^{n-1}\overline{P_n\overline{v}}$.
\end{proposition}

\begin{proof}[Proof]
We tackle with the case when $v=v_{i_1}\cdots v_{i_n}$ is a monomial. In this case, the identity to be proved is $\Delta_nP_nv=(-1)^{n-1}\overline{P_nv}$.
\par
For $1\leq j_1<\cdots<j_s\leq n-1$, we denote
$$E_{j_1,\cdots,j_s}=(\s_{n-1}\cdots \s_{j_1})\cdots (\s_{n-1}\cdots \s_{j_s}),$$
then the Dynkin operator $P_n$ can be written as:
$$P_n=\sum_{s=0}^{n-1}(-1)^s \sum_{1\leq j_1<\cdots <j_s\leq n-1}E_{j_1,\cdots ,j_s}.$$

\begin{lemma}
Let $1\leq j_1<\cdots <j_s\leq n-1$ and $1\leq j_1'<\cdots <j_t'\leq n-1$ such that $\{j_1,\cdots,j_s\}$ and $\{j_1'\cdots,j_t'\}$ form a partition of the set $\{1,\cdots,n-1\}$. Then for any $v=v_{i_1}\cdots v_{i_n}$, 
$$\Delta_nE_{j_1,\cdots,j_s}v=\overline{E_{j_1'\cdots,j_t'}v}.$$ 
\end{lemma}

\begin{proof}
To simplify notations, we define 
$$Q_{i_1,\cdots,i_n}^{j_1,\cdots,j_s}=q_{j_s,j_s+1}\cdots q_{j_s,i_n}q_{j_{s-1},j_{s-1}+1}\cdots q_{j_{s-1},i_n}\cdots q_{j_1,j_1+1}\cdots q_{j_1,i_n}.$$
Then the condition $\theta_nv=v$ and the fact that the braiding matrix is symmetric imply
$$Q_{i_1,\cdots,i_n}^{j_1,\cdots,j_s}=\overline{Q_{i_1,\cdots,i_n}^{j_1',\cdots,j_t'}}.
$$
With this notation, 
$$\Delta_n E_{j_1,\cdots,j_s}(v_{i_1}\cdots v_{i_n})=Q_{i_1,\cdots,i_n}^{j_1,\cdots,j_s}v_{j_1}\cdots v_{j_s}v_{j_t'}\cdots v_{j_1'},$$
$$E_{j_1',\cdots,j_t'}(v_{i_1}\cdots v_{i_n})=Q_{i_1,\cdots,i_n}^{j_1',\cdots,j_t'}v_{j_1}\cdots v_{j_s}v_{j_t'}\cdots v_{j_1'}.$$
The lemma can be proved by combining the above formulas.
\end{proof}
We compute the left hand side when $v=v_{i_1}\cdots v_{i_n}$,
\begin{eqnarray*}
\Delta_nP_nv &=& \sum_{s=0}^{n-1}(-1)^s\sum_{1\leq j_1<\cdots<j_s\leq n-1}\Delta_nE_{j_1,\cdots,j_s}v\\
&=& \sum_{s=0}^{n-1}(-1)^s\sum_{1\leq j_1'<\cdots<j_{n-1-s}'\leq n-1}\overline{E_{j_1',\cdots,j_{n-1-s}'}v}\\
&=& (-1)^{n-1}\sum_{t=0}^{n-1}(-1)^t \sum_{1\leq j_1'<\cdots<j_t'\leq n-1}\overline{E_{j_1',\cdots,j_t'}v}=(-1)^{n-1}\overline{P_nv}.
\end{eqnarray*}
It remains to tackle the general case: suppose that $v=\sum a_i v_i$ where $v_i$ are monomials. Applying the above formula gives:
\begin{eqnarray*}
\Delta_nP_nv &=& \sum a_i\Delta_nP_nv_i\\
&=& (-1)^{n-1}\sum a_i\overline{P_n v_i}\\
&=& (-1)^{n-1}\sum \overline{\overline{a_i}P_nv_i}=(-1)^{n-1}\overline{P_n\overline{v}}.
\end{eqnarray*}
\end{proof}

\begin{corollary}
If $v\in T(V)$ satisfying $\overline{v}=v$ and $P_nv$ is a right pre-relation. Then $\overline{P_nv}$ is a left pre-relation.
\end{corollary}

For instance, $\overline{v}=v$ holds when $v$ is a monomial.

\section{On the specialization problem}\label{Sec5.7}

Recall that the field $\mc$ is of characteristic $0$.

\subsection{A result on Kac-Moody Lie algebras}

Let $C$ be an arbitrary matrix. The Kac-Moody Lie algebra associated to $C$ is defined by: $\g(C)=\widetilde{\g}(C)/\mathfrak{r}$, where $\widetilde{\g}(C)$ is the Lie algebra with Chevalley generators $e_i,f_i,h_i$ for $i\in I$ and relations with respect to a realization $(\mathfrak{h},\Pi,\Pi^\vee)$ of $C$ and $\mathfrak{r}$ is the unique maximal ideal in $\widetilde{\g}(C)$ intersecting $\mathfrak{h}$ trivially (see \cite{Kac90}, Chapter 1 for details). Moreover, we have the following decomposition as subalgebras
$$\widetilde{\g}(C)=\widetilde{\mathfrak{n}}_-\oplus\mathfrak{h}\oplus\widetilde{\mathfrak{n}}_+.$$ 
It gives $\mathfrak{r}=\mathfrak{r_+}\oplus\mathfrak{r}_-$ as a direct sum of ideals, where $\mathfrak{r}_+=\mathfrak{r}\cap\wt{\mathfrak{n}}_+$ and $\mathfrak{r}_-=\mathfrak{r}\cap\wt{\mathfrak{n}}_-$. We denote the quotients by $\mathfrak{n}_-=\widetilde{\mathfrak{n}}_-/\mathfrak{r}_-$ and $\mathfrak{n}_+=\widetilde{\mathfrak{n}}_+/\mathfrak{r}_+$.
\par
For these Lie algebras, we let $U(\mathfrak{r}_\pm)$, $U(\mathfrak{n}_\pm)$, $U(\widetilde{\mathfrak{n}}_\pm)$, $U(\mathfrak{g})$ and $U(\widetilde{\g})$ denote the corresponding enveloping algebras. The following theorem and proposition should be known by experts in enveloping algebras. We provide their proofs for the absence of a proper reference.

\begin{theorem}\label{Thm:KM}
Let $x\in U(\wt{\mathfrak{n}}_-)$. There exists an equivalence between:
\begin{enumerate}
\item for any $i\in I$, $[e_i,x]\in U(\mathfrak{r}_-)$;
\item $x\in U(\mathfrak{r}_-)$.
\end{enumerate}
\end{theorem}

The proof of this theorem occupies the rest of this subsection.

\begin{proposition}\label{Prop:constant}
Let $x\in U(\mathfrak{n}_-)$ such that for any $i\in I$, $[e_i,x]\in\mc$. Then $x\in\mc$ is a constant.
\end{proposition}

\begin{proof}
Given $x\in U(\mathfrak{n}_-)$ not be a constant, we search for an index $t\in I$ such that $[e_t,x]\notin\mc$.
\par
We start by showing that it suffices to consider those $x$ which are homogeneous with respect to the height gradation. Write $x=x_l+\cdots+x_1+x_0$ such that $x_i$ is of height $i$, since $[e_k,x_s]$ is either of height $s-1$ or of height $0$, to show that $[e_t,x]\notin\mc$, it suffices to consider $[e_t,x_l]$.
\par
We apply induction on the height: if $x$ is of height $1$, then it is in the vector space generated by $f_i$ for $i\in I$. This case is clear.
\par
Suppose that a totally ordered basis $\{f_\gamma\}_{\gamma\in\Gamma}$ of $\mathfrak{n}_-$ is chosen such that elements of less heights are smaller. Let $f_\beta$ denote the maximal basis element among those appearing in monomials of $x$. By PBW theorem,
$x=\sum_{s=0}^l r_sf_\beta^s$
where $f_\beta$ does not appear in monomials of $r_s$. Hence
$$[e_i,x]=\sum_{s=0}^l [e_i,r_s]f_\beta^s+\sum_{s=0}^l r_s[e_i,f_\beta^s]$$
where the term containing $f_\beta^l$ is $[e_i,r_l]$. There are three cases:
\begin{enumerate}
\item $r_l\in\mc$ and all $r_s$ are zero. It suffices to apply Lemma 1.5 in \cite{Kac90}.
\item $r_l\in\mc$ and there exists some maximal $0\leq k\leq l-1$ such that $r_k\neq 0$. In this case, $r_{k}\notin\mc$ since $x$ is homogeneous and the highest power in $f_\beta$ in $[e_i,x]$ is $[e_i,r_{l-1}+lf_\beta]f_\beta^{l-1}$ if $k=l-1$ and otherwise is $l[e_i,f_\beta]f_\beta^{l-1}$. For the former, by induction we can always find some $e_t$ such that $[e_t,r_{l-1}+lf_\beta]\notin\mc$; for the latter it suffices to apply Lemma 1.5 in \cite{Kac90}.
\item $r_l\notin\mc$. Since $r_s$ is homogeneous and of a smaller height, the induction can be applied to give some $e_t$ such that $[e_t,r_l]\neq 0$, therefore $[e_t,x]\neq 0$.
\end{enumerate}
\end{proof}

\begin{proof}[Proof of theorem]
(2) implies (1) is clear since $[e_i,\cdot]$ is a derivation and $\mathfrak{r}_-$ is an ideal in $\widetilde{\mathfrak{n}}_-$.
\par
We suppose that (1) holds and introduce another $\mathbb{N}$-gradation on $U(\wt{\mathfrak{n}}_-)$: by PBW theorem, $U(\wt{\mathfrak{n}}_-)$ is a free $U(\mathfrak{r}_-)$-module (see \cite{Dixmier}, Proposition 2.2.7). We define the partial degree on $U(\wt{\mathfrak{n}}_-)$ by taking the height gradation on $U(\mathfrak{r}_-)$ and letting elements in $U(\mathfrak{n}_-)$ be of degree $0$. Then $x\in U(\wt{\mathfrak{n}}_-)$ is of partial degree $0$ if and only if $x\in U(\mathfrak{n}_-)$.
\par
The proof will be effectuated by induction on the biggest partial degree $l$ among components of $x$.
\par
If $l=0$, $x$ is in $U(\mathfrak{n}_-)$, $[e_i,x]\in U(\mathfrak{r}_-)$ implies that $[e_i,x]\in\mc$. By Proposition \ref{Prop:constant}, $x$ is a constant therefore in $U(\mathfrak{r}_-)$.
\par
In general, let $l$ be the maximal partial degree among components of $x$, we write $x=x_l+x_{l-1}+\cdots+x_1+x_0$ where $x_i\in U(\wt{\mathfrak{n}}_-)$ is of partial degree $i$. We write $x_l=\sum r_kn_k$ for some $r_k\in U(\mathfrak{r}_-)$ and $n_k\in U(\mathfrak{n}_-)$ such that these $r_k$ are linearly independent. In $[e_i,x]$, since the partial degree of $[e_i,r_k]$ is less than that of $r_k$, the component of maximal partial degree is given by $\sum r_k[e_i,n_k]$; as $[e_i,x]\in U(\mathfrak{r}_-)$, it forces $[e_i,n_k]\in U(\mathfrak{r}_-)$ and hence $[e_i,n_k]\in\mc$. This shows that for any $i\in I$ and any $k$, $[e_i,n_k]\in\mc$; by Proposition \ref{Prop:constant}, $n_k$ are constants and $x_l\in U(\mathfrak{r}_-)$. Finally we consider $x-x_l$: it has lower partial degree and satisfies $[e_i,x-x_l]\in U(\mathfrak{r}_-)$. By induction hypothesis, $x-x_l\in U(\mathfrak{r}_-)$, hence $x\in U(\mathfrak{r}_-)$.
\end{proof}

If moreover $C$ is a generalized Cartan matrix, some elements in $\mathfrak{r}$ are discovered in Section 3.3 of \cite{Kac90}: in $\mathfrak{g}(C)$, for $i\neq j$,
\begin{equation}\label{Serre}
(\mathrm{ad}e_i)^{1-c_{ij}}(e_j)=0,\, \, \,  (\mathrm{ad}f_i)^{1-c_{ij}}(f_j)=0.
\end{equation}
If the matrix $C$ is not symmetrizable, the ideal generated by these relations may not exhaust $\mathfrak{r}$.

\subsection{Specialization (I): general definition and counterexample.}

We follow the specialization procedure in \cite{HK02}. Let $C$ be a generalized Cartan matrix, $\mathcal{A}=\mc[q^{\frac{1}{2}},q^{-\frac{1}{2}}]$ and $\mathcal{A}_1$ be the localization of $\mc[q^{\frac{1}{2}}]$ at $(q^{\frac{1}{2}}-1)$.
\par
When the braiding matrix $A$ is of form $(q^{c_{ij}})_{1\leq i,j\leq N}$ for a generalized Cartan matrix $C=(c_{ij})_{1\leq i,j\leq N}$, we will denote the Hopf algebra $T_q(A)$ by $T_q(C)$ and $D_q(A)$ by $D_q(C)$.
\par
We start by defining an $\mathcal{A}_1$-form of $T_q(C)$: let $T_{\mathcal{A}_1}(C)$ be the $\mathcal{A}_1$-subalgebra of $T_q(C)$ generated by 
$$E_i,\ \  F_i,\ \  K_i^{\pm 1}\ \  \mathrm{and}\ \ [K_i;0]=\frac{K_i-K_i^{-1}}{q-q^{-1}}$$
for any $i\in I$. It inherits a Hopf algebra structure from that of $T_q(C)$. We let $T_{\mathcal{A}_1}^{<0}(C)$ (resp. $T_{\mathcal{A}_1}^{>0}(C)$) denote the subalgebra of $T_{\mathcal{A}_1}(C)$ generated by $F_i$ (resp. $E_i$) for $i\in I$.
\par
Since $(q^{\frac{1}{2}}-1)$ is a maximal ideal in $\mathcal{A}_1$, $\mc$ admits an $\mathcal{A}_1$-module structure via $\mathcal{A}_1/(q^{\frac{1}{2}}-1)\cong \mc$, given by evaluating $q^{\frac{1}{2}}$ to $1$. We define $T_1(C)=T_{\mathcal{A}_1}(C)\ts_{\mathcal{A}_1}\mc$, there exists a natural algebra morphism $\widetilde{\s}:T_{\mathcal{A}_1}(C)\ra T_1(C)$, which is called the specialization map.
\par
For $i\in I$, we let $e_i$, $f_i$ and $h_i$ denote the images of $E_i$, $F_i$ and $\frac{K_i-1}{q-1}$ under the map $\wt{\s}$. Then $K_i^{\pm 1}$ are sent to $1$ and $[K_i;0]$ has image $h_i$ under $\wt{\s}$. Relations in $T_{\mathcal{A}_1}(C)$ are specialized to relations in $T_1(C)$;
$$[e_i,f_j]=\wt{\s}([E_i,F_j])=\delta_{ij}\wt{\s}\left([K_i;0]\right)=\delta_{ij}h_i;$$
$$[h_i,e_j]=\wt{\s}([[K_i;0],E_j])=\wt{\s}\left(\frac{(1-q^{-c_{ij}})K_i-(1-q^{c_{ij}})K_i^{-1}}{q-q^{-1}}E_j\right)=c_{ij}e_j;$$
and similarly $[h_i,f_j]=-c_{ij}f_j$ and $[h_i,h_j]=0$.
\par
The following facts hold:
\begin{enumerate}
\item The specialization map $\widetilde{\s}:T_{\mathcal{A}_1}(C)\ra T_1(C)\cong U(\wt{\g}(C))$ is a Hopf algebra morphism. When composed with the projection $U(\wt{\g}(C))\ra U(\g(C))$, it gives a Hopf algebra morphism $\s:T_{\mathcal{A}_1}(C)\ra U(\g(C))$, which is also called the specialization map.
\item The restrictions of $\s$ give the following specialization maps $T_{\mathcal{A}_1}^{<0}(C)\ra U(\mathfrak{n}_-(C))$ and $T_{\mathcal{A}_1}^{>0}(C)\ra U(\mathfrak{n}_+(C))$.
\end{enumerate}

To obtain a true specialization map of the quantum group, the morphism $\s$ should pass through the quotient by defining ideals.

\begin{example}\label{Ex:main}
We consider the following non-symmetrizable generalized Cartan matrix
$$C=\begin{bmatrix} 2 & -2 & -1 \\ -1 & 2 & -1\\ -3 & -1 & 2\end{bmatrix}.$$
In the braided tensor Hopf algebra of diagonal type associated to this matrix, we want to find some particular pre-relations: it is easy to show that $\theta_4(F_3^3F_1)=F_3^3F_1$. Recall that $T_4'=(1-\s_3^2\s_2\s_1)(1-\s_3^2\s_2)(1-\s_3^2)$; since $1-\s_3^2\s_2$ and $1-\s_3^2$ act as non-zero scalars on $F_3^3F_1$ and $1-\s_3^2\s_2\s_1$ acts as $0$ on it,
$$T_4P_4(F_3^3F_1)=T_4'(F_3^3F_1)=0.$$
Moreover, since $\iota_3^4(T_3')=(1-\s_3^2\s_2)(1-\s_3^2)$ acts as a non-zero scalar on $F_3^3F_1$, by definition, $P_4(F_3^3F_1)$ is a right pre-relation of degree $4$ where:
$$P_4(F_3^3F_1)=F_3^3F_1-(q^{-3}+q^{-1}+q)F_3^2F_1F_3+(q^{-4}+q^{-2}+1)F_3F_1F_3^2-q^{-3}F_1F_3^3.$$
If the specialization map to the enveloping algebra of the Kac-Moody Lie algebra associated to $C$ were well-defined, this element would be specialized to 
$$[f_3,[f_3,[f_3,f_1]]]=f_3^3f_1-3f_3^2f_1f_3+3f_3f_1f_3^2-f_1f_3^3$$ 
in $U(\wt{\mathfrak{n}}_-)$. We show that it is not contained in $U(\mathfrak{r}_-)$ so does not give $0$, contradicts the definition of the Kac-Moody Lie algebra.
\par
The successive adjoint actions of $e_3$ give:
$$[e_3,f_3^3f_1-3f_3^2f_1f_3+3f_3f_1f_3^2-f_1f_3^3]=3(f_3^2f_1-2f_3f_1f_3+f_1f_3^2),$$
$$[e_3,f_3^2f_1-2f_3f_1f_3+f_1f_3^2]=4(f_3f_1-f_1f_3),$$
$$[e_3,f_3f_1-f_1f_3]=3f_1.$$
If $[f_3,[f_3,[f_3,f_1]]]$ were in $U(\mathfrak{r}_-)$, so is $f_1$ hence $[e_1,f_1]=h_1$ according to Theorem \ref{Thm:KM}. This is impossible since by definition of the Kac-Moody Lie algebra, $\mathfrak{r}_-\cap \mathfrak{h}=\{0\}$.
\par
As a conclusion, this example shows that the specialization map may not be well-defined if the matrix is not symmetric.
\end{example}

\subsection{Specialization (II): the quantum group case.}
Let $C$ be a generalized Cartan matrix and $\overline{C}$ be the associated averaged matrix. To have a well-defined specialization map, we need to pass to the $q$-enveloping algebra associated to this averaged matrix. We suppose moreover that the matrix $C$ is non-symmetrizable as otherwise there would be no problem.
\par
Recall that $U_q(C):=D_q(\overline{C})$ is the quotient of $T_q(\overline{C})$ by its defining ideals and $U_q^{<0}(C)$ is the subalgebra of $U_q(C)$ generated by $F_i$ for $i\in I$. $U_q(C)$ admits an $\mathcal{A}_1$-form since the defining ideals are given by the kernel of the total symmetrization map $S_n$, which preserves both $T_{\mathcal{A}_1}^{<0}(\overline{C})$ and $T_{\mathcal{A}_1}^{>0}(\overline{C})$. This $\mathcal{A}_1$-form of $U_q(C)$ is generated as an $\mathcal{A}_1$-module by $E_i$, $F_i$, $K_i^{\pm 1}$ and $[K_i;0]$ for $i\in I$, we let $U_{\mathcal{A}_1}(C)$ denote it. Since operators $T_n$, $P_n$, $T_n'$, $U_n$, $Q_n$, $U_n'$, $\Delta_n$ and $\theta_n$ preserve the subalgebra $T_{\mathcal{A}_1}(\overline{C})$ of $T_q(\overline{C})$, left and right pre-relations, left and right constants are well-defined over $\mathcal{A}_1$.

\begin{theorem}\label{Thm:3}
The specialization map $\s:T_{\mathcal{A}_1}(\overline{C})\ra U(\g(\overline{C}))$ passes the quotient to give a surjective map $\s:U_{\mathcal{A}_1}(C)\ra U(\g(\overline{C}))$.
\end{theorem}

The proof of this theorem will occupy the rest of this subsection. We start by the following lemma (see also Lemma 4.15 in \cite{H10}).

\begin{lemma}
For any $w\in T_{\mathcal{A}_1}^{<0}(\overline{C})$ and any $i\in I$,
$$[E_i,w]=\frac{K_i\p_i^L(w)-\p_i^R(w)K_i^{-1}}{q-q^{-1}}=\frac{{\dd}_i^R(w)K_i-\p_i^R(w)K_i^{-1}}{q-q^{-1}}\in T_{\mathcal{A}_1}^{<0}(\overline{C}).$$
\end{lemma}

This formula can be proved either by induction or by verifying directly on a monomial; notice that the symmetric condition on the braiding matrix is necessary.
\par
Recall that $\widetilde{\s}:T_q(\overline{C})\ra U(\widetilde{\g}(\overline{C}))$ is the specialization map.

\begin{lemma}
Let $w\in T_{\mathcal{A}_1}^{<0}(\overline{C})$ be a right constant of degree $n$. Then $\widetilde{\s}(w)\in U(\mathfrak{r}_-)$.
\end{lemma}

\begin{proof}
By Theorem \ref{Thm:KM}, to show $\widetilde{\s}(w)\in U(\mathfrak{r}_-)$, it suffices to verify that for any $i\in I$, $$\wt{\s}([E_i,w])=[e_i,\wt{\s}(w)]\in U(\mathfrak{r}_-).$$
We apply induction on the degree $n$ of the right constant $w$.
\par
The case $n=2$ is clear since all constants of degree $2$ are computed in Proposition \ref{Prop:deg2}. For general $n\geq 3$, by the above lemma and the fact that $\p_i^R(w)=0$ for any $i\in I$,
$$[E_i,w]=K_i\p_i^L\left(\frac{w}{q-q^{-1}}\right)\in T_{\mathcal{A}_1}^{<0}(\overline{C}).$$
Since $\p_i^L$ and $\p_i^R$ commute, $\p_i^L(w)$ is annihilated by $T_{n-1}$ and of degree at most $n-1$. By induction hypothesis, $\p_i^L(\frac{w}{q-q^{-1}})$ is specialized to $U(\mathfrak{r}_-)$, hence
$$\wt{\s}([E_i,w])=\wt{\s}\left(K_i\p_i^L\left(\frac{w}{q-q^{-1}}\right)\right)\in U(\mathfrak{r}_-).$$
\end{proof}

\begin{proof}[Proof of theorem]
We have proved in the above lemma that right constants are specialized to $U(\mathfrak{r}_-)$ under $\widetilde{\s}$. A similar argument can be applied to left constants to show that their specializations are in $U(\mathfrak{r}_+)$. We obtain therefore a well-defined algebra map $\s:U_{\mathcal{A}_1}(C)\ra U(\g(\overline{C}))$ and the surjectivity is clear.
\end{proof}

\subsection{Specialization (III): the Nichols algebra case.}
Let $C$ be a generalized Cartan matrix, $A=(q^{c_{ij}})_{1\leq i,j\leq N}$ and $D^{<0}(C)$ be the Nichols algebra of braiding matrix $A$ with respect to a basis $F_1,\cdots,F_N$, which is the subalgebra of $D^{\leq 0}(A)$ generated by $F_i$ for $i\in I$.

\begin{theorem}
There exists a surjective algebra morphism $\vp:D^{<0}(C)\ra U(\mathfrak{n}_-(\overline{C}))$ sending $v_i$ to $f_i$.
\end{theorem}

\begin{proof}
We let $D^{<0}(\overline{C})$ denote the Nichols algebra of diagonal type of braiding matrix $(q^{\overline{c}_{ij}})_{1\leq i,j\leq N}$ with respect to basis $w_1,\cdots,w_N$. By Proposition \ref{Prop:AS}, there exists a linear isomorphism $\psi:D^{<0}(C)\ra D^{<0}(\overline{C})$ sending $F_i$ to $w_i$. Composing with the restriction of the specialization map $\s$ to the negative part of $U_q(C)$ gives linear surjection $\vp:D^{<0}(C)\ra U(\mathfrak{n}_-(\overline{C}))$.
\par
It remains to show that $\vp$ is an algebra morphism: for $1\leq i,j\leq N$, when $i\leq j$, 
$$\vp(F_iF_j)=\s\circ\psi(F_iF_j)=\s(q^{\frac{1}{2}(c_{ji}-c_{ij})}w_iw_j)=f_if_j;$$
if $i>j$,
$$\vp(F_iF_j)=\s\circ\psi(F_iF_j)=\s(w_iw_j)=f_if_j.$$
\end{proof}

\section{Application}\label{Sec5.8}

It is natural to ask for the size of $\Rel_r$, we will relate it to the integral points of some quadratic forms.

\subsection{General calculation}
Let $A=(a_{ij})_{1\leq i,j\leq N}\in M_N(\mathbb{Z})$ be a generalized Cartan matrix. We consider the element $v_{\underline{i}}$ for $\underline{i}=(1^{m_1},\cdots,N^{m_N})$: the action of the central element $\theta_m$ where $m=m_1+\cdots+m_N$ gives
$$\theta_m(v_{\underline{i}})=q^\lambda v_{\underline{i}}$$
where
$$\lambda=\sum_{k=1}^N 2m_k(m_k-1)-\sum_{p=1}^N\sum_{q<p}(a_{pq}+a_{qp})m_pm_q.$$
So there exists a pre-relation in $\mc[X_{\underline{i}}]$ only if $\lambda=0$. To find these pre-relations, it suffices to consider the integral solutions of this quadratic form.

\subsection{Study of the quadratic form}
The above computation motivates the study the following quadratic forms:
$$Q(x_1,\cdots,x_n)=\sum_{i=1}^n x_i^2-\sum_{i<j}b_{ij}x_ix_j,$$
$$S(x_1,\cdots,x_n)=\sum_{i=1}^n(x_i-1)^2,$$
where $b_{ij}=a_{ij}+a_{ji}$ are non-negative integers as in the last subsection. 
\par
Let $m\leq n$ be an integer (not necessary positive) and $C_m$ be the intersection of the following two varieties 
$$Q(x_1,\cdots,x_n)=m,\,\,\, S(x_1,\cdots,x_n)=n-m.$$
Let $E(C_m)$ be the set of integral points on $C_m$ and $E=\bigcup_{m\leq n}E(C_m)$. Then the set of all integral solutions of $\lambda=0$ is the same as $E$.

\begin{proposition}
If the quadratic form $Q(x_1,\cdots,x_n)$ is positive semi-definite, $E$ is a finite set.
\end{proposition}

\begin{proof}[Proof]
If $Q(x_1,\cdots,x_n)$ is semi-positive definite, $E$ is a finite union of $E(C_m)$ for $0\leq m\leq n$. For each $m$, as $S(x_1,\cdots,x_n)=n-m$ is compact, so is its intersection with $Q(x_1,\cdots,x_n)=m$. The finiteness of $E(C_m)$ and $E$ are clear. 
\end{proof}

\begin{corollary}
If the quadratic form $Q(x_1,\cdots,x_n)$ is semi-positive definite, the defining ideal $\mathfrak{I}(V)$ is finitely generated.
\end{corollary}

\begin{proof}[Proof]
By above proposition, there are only a finite number of indices $\underline{i}$ such that $\mc[X_{\underline{i}}]$ contains right pre-relations; moreover, each $\mc[X_{\underline{i}}]$ is finite dimensional. 
\end{proof}

\end{document}